\documentclass[12pt, reqno]{amsproc}
\usepackage{graphicx, enumerate}
\usepackage[margin=1in]{geometry}
\usepackage{amsmath}
\usepackage{amsfonts}
\usepackage{amssymb}
\usepackage{amsfonts,amsmath,amssymb}
\usepackage{amsthm}
\newtheorem{lemma}{Lemma}
\newtheorem{theorem}{Theorem}

\newtheorem{eg}{Example}
\title[Infinite Order of Growth of Solutions of Second Order Linear Differential Equations]{Infinite Order of Growth of Solutions of Second Order Linear Differential Equations}
\author[Naveen Mehra and V. P. Pande]{Naveen Mehra and V. P. Pande}
\address{Naveen Mehra, Department of Mathematics, Kumaun University, S.S.J Campus, Almora-263601, Uttarakhand, India}
\email{naveenmehra00@gmail.com}
\address{V. P. Pande, Department of Mathematics, Kumaun University, S.S.J Campus, Almora-263601, Uttarakhand, India}
\email{vijpande@gmail.com}
\thanks{The research work of the first author is supported by research fellowship from Council of Scientific and Industrial Research (CSIR), New Delhi.} 
\begin{document}
\maketitle
\begin{abstract}
Considering differential equation $f''+A(z)f'+B(z)f=0$, where $A(z)$ and $B(z)$ are entire complex functions, our results revolve around proving all non-trivial solutions are of infinite order taking various restrictions on coefficients $A(z)$ and $B(z)$.   
\end{abstract}

\section{Introduction and statement of main results}
We consider the second order homogeneous linear differential equation,
\begin{equation}\label{cde1}
f''+A(z)f'+B(z)f=0
\end{equation}
where $A(z)$ and $B(z)$ are entire functions. It is well known result that all solutions of \eqref{cde1} are entire functions(see \cite{herold}). All solutions are of finite order if and only if $A(z)$ and $B(z)$ are polynomial(see \cite{herold}). Thus, if at least one of the coefficients are not polynomial then solutions can be of infinite order. It is wide area of research to find conditions on $A(z)$ and $B(z)$ such that solutions are of infinite order. There are many results concerning this problem. Following theorem is the collection of such basic results which provide all non-trivial solutions of infinite order.
 \begin{theorem}
 Suppose $A(z)$ and $B(z)$ are non-constant entire functions, then all non-trivial solutions of \eqref{cde1} are of infinite order, if one of the following holds:
\begin{enumerate}[(i)]
\item\cite{gundersen2} $\rho(A) <\rho(B)$
\item\cite{gundersen2} $A(z)$ is a polynomial and $B(z)$ is transcendental
\item\cite{heller} $\rho(B)<\rho(A)\leq\frac{1}{2}$.
\end{enumerate}
\end{theorem}
 Consider the equation
\begin{equation}\label{cde2}
w''+P(z)w=0
\end{equation}
where $P(z)$ is a polynomial of degree $n$. There are several papers in  which the authors have considered $A(z)$ to be a solution of equation \eqref{cde2} and $B(z)$ to satisfy different conditions so that all  non-trivial solutions are of infinite order. The next theorem is the collection of all these results.
\begin{theorem}
Suppose $A(z)$ is a solution of \eqref{cde2} and $B(z)$ is a transcendental entire  function satisfying one of the conditions mentioned below. Then all non-trivial solutions of \eqref{cde2} are of infinite order.
\begin{enumerate}[(i)]
\item\cite{wu&wu} $\rho(B)<\frac{1}{2}$
\item\cite{long&qiu} $\mu (B) < \frac{1}{2}$ and $\rho(A)\neq\rho(B)$
\item\cite{wlh} $\mu (B) < \frac{1}{2} + \frac{1}{2(n+1)}$ and $\rho(A)\neq\rho(B)$
\item\cite{long} $B(z)$ has Fabry gap and $\rho(A)\neq\rho(B)$.
\end{enumerate}
\end{theorem}
 Motivated by above results we have replaced the conditions on $B(z)$ to be transcendental entire function satisfying $T(r,B)\sim \log M(r,B)$ in a set $E$ of positive upper logarithmic density, where the notation $T(r,B)\sim \log M(r,B)$ means $$\lim\limits_{r\to\infty}\frac{T(r,B)}{\log M(r,B)}=1.$$
\begin{theorem}
Let $A(z)$ be a non-trivial solution of \eqref{cde2}, where $P(z)=a_nz^n+\cdots +a_0$, $a_n\neq 0$ and $B(z)$ be transcendental entire function satisfying $T(r,B)\sim \log M(r,B)$ in a set $E$ of positive upper logarithmic density, then all non-trivial solutions of \eqref{cde1} are of infinite order.
\end{theorem}
Kwon in his paper \cite{kwon} considered $\rho(A)>1$ of finite non-integral order with all its zeros fixed in a sector and $0<\rho(B)<1/2$. Kumar et. al. \cite{myself} assumed that $B(z)$ has Fabry gap. Recall that, for an entire function $\displaystyle{f(z)=\sum\limits_{n=0}^{\infty} a_{\lambda_n}z^{\lambda_n}}$ if $\displaystyle{\sum\limits_{n=0}^{\infty}\frac{1}{\lambda_n}}$ diverges to $\infty$ we say that it has Fabry gap.
\begin{theorem}
Suppose that $A(z)$ is an entire function of finite non-integral order with $\rho(A)>1$, and that all the zeros of $A(z)$ lies in the angular sector  $\theta_1<argz<\theta_2$ satisfying
$$\theta_2-\theta_1<\frac{\pi}{p+1}$$ if $p$ is odd, and
$$\theta_2-\theta_1<\frac{(2p-1)\pi}{2p(p+1)}$$ if $p$ is even, where $p$ is the genus of $A(z)$. Let $B(z)$ be an entire function satisfying the conditions mentioned below. Then all non-trivial solutions $f$ of equation \eqref{cde1} are of infinite order.
\begin{enumerate}[(i)]
\item\cite{kwon} $0<\rho(B)<\frac{1}{2}$,
\item\cite{myself} $B(z)$ has Fabry gap.
\end{enumerate}
\end{theorem}
Motivated by above results we have replaced $B(z)$ to be a transcendental entire function having a multiply connected Fatou component.
\begin{theorem}
Suppose that $A(z)$ be an entire function of finite non-integral order with $\rho(A)>1$, and that all the zeros of $A(z)$ lies in the angular sector  $\theta_1<argz<\theta_2$ satisfying
$$\theta_2-\theta_1<\frac{\pi}{p+1}$$ if $p$ is odd, and
$$\theta_2-\theta_1<\frac{(2p-1)\pi}{2p(p+1)}$$ if $p$ is even, where $p$ is the genus of $A(z)$ and let $B(z)$ be a transcendental entire function with a multiply connected Fatou component. Then all non-trivial solutions of equation \eqref{cde1} are of infinite order.
\end{theorem}
In several papers like \cite{kks}, \cite{kumarsaini}, \cite{saini}, \cite{wanglaine} and \cite{wlh} authors have considered $A(z)=h(z)e^{P(z)}$, where $P(z)$ is a polynomial of degree $n$ satisfying $\lambda(A)<\rho(A)$ and $B(z)$ with various conditions. Here we have exchanged the conditions of $A(z)$ with $B(z)$ and prove the following result. 
\begin{theorem}
Let $B(z)=h(z)e^{P(z)}$ be a transcendental entire function satisfying $\lambda(B)<\rho(B)=n$, where $P(z)$ is a polynomial of degree $n$ and $A(z)$ satisfies the properties mentioned below, then all non-trivial solutions of \eqref{cde1} are of infinite order.
\begin{enumerate}[(a)]
\item $A(z)$ has Fabry gap
\item $A(z)$ satisfies $T(r,A)\sim \log M(r,A)$ in a set $E$ of positive upper logarithmic density
\item $A(z)$ has multiply connected Fatou component.
\end{enumerate}
\end{theorem}
Following example shows that if we skip conditions of our theorem then we do get solution of finite order.
\begin{eg}\begin{enumerate}[(a)]
\item 
$f''-e^zf'+e^zf=0$ has a solution $e^z-1$ of finite order $1$.\\
$T(r,e^z)=\frac{r}{\pi}$ and $\log M(r,e^z)=r$,thus $T(r,e^z)\not\sim \log M(r,e^z)$
\begin{enumerate}[(i)]
\item
$A(z)=-e^z$ is a solution of $w''-w=0$ and $B(z)=e^z$ does not satisfy the condition of Theorem 3.  
\item
$B(z)=e^z$ satisfy $\lambda(B)<\rho(B)$ but $A(z)$ does not satisfy the condition of Theorem 6(b)
\end{enumerate}
\item 
$f''+(e^z-1)f'+e^zf=0$ has a solution $e^z$ of finite order $1$.\\
$B(z)=e^z$ satisfy $\lambda(B)<\rho(B)$ but $A(z)$ does not satisfies the condition of Theorem 6(a) i.e. it does not has Fabry gap.
\item
$f''+Cz^2 \displaystyle{ \prod_{n=1}^{\infty}\left(1+\frac{z}{a_n}\right) }f' -Cz \displaystyle{ \prod_{n=1}^{\infty}\left(1+\frac{z}{a_n}\right)f=0 }$ has a solution $z$ of order $0$.
$$A(z)=Cz^2 \displaystyle{ \prod_{n=1}^{\infty}\left(1+\frac{z}{a_n}\right) }$$ where the $a_n$ satisfy $1 < a_1 < a_2 < ...$ and grow so rapidly that $a_{n+1} < A(a_n) <2a_{n+1}$, is constructed by Baker\cite{baker}. It has multiply connected Fatou component but $\lambda(B)=\rho (B)$ i.e. it does not satisfy the condition of Theorem 6(c).
\item 
$f''+Cz^3 \displaystyle{ \prod_{n=1}^{\infty}\left(1+\frac{z}{a_n}\right) }f' -Cz^2 \displaystyle{ \prod_{n=1}^{\infty}\left(1+\frac{z}{a_n}\right)f=0 }$ has a solution $z$ of order $0$.\\
$B(z)$ has multiply connected Fatou component but $\rho(A)\leq 1$ i.e. it does not satisfy the condition of Theorem 5.
\end{enumerate}
\end{eg}

\section{Preliminary Lemma}
Next lemma is due to Gundersen\cite{gundersen} and it gives an estimation of logarithmic derivatives.  
\begin{lemma}\cite{gundersen}\label{gundersen} Let $f$ be a transcendental entire function of finite order $\rho$, let $\Gamma = \lbrace(k_1,j_1), (k_2, j_2) . . . (k_m, j_m)\rbrace$ denote finite set of distinct pairs of integers that satisfies $k_i > j_i \geq 0$, for $i = 1, 2, . . .m$, and let
$\epsilon > 0$ be a given constant. Then the following three statements holds:\begin{enumerate}
\item[(i)] there exists a set $E_1 \subset [0, 2\pi)$ that has linear measure zero, such that if $\psi_0 \in [0, 2\pi) -E_1$, then there is a constant $R_0 = R_0(\psi_0) > 0$ so that for all z satisfying $arg z = \psi_0$ and $|z| \geq R_0$ and for all $(k, j) \in \Gamma$, we have
\begin{equation}
\left|\frac{f^{(k)}(z)}{f^{(j)}(z)}\right| \leq |z|^{(k-j)(\rho-1+\epsilon))}
\end{equation}\label{f"byf}
\item[(ii)] there exists a set $E_2 \subset (1,\infty)$ that has finite logarithmic measure, such that for all z satisfying $|z|$ does not belongs to $E_2 \cup[0, 1]$ and for all $(k, j) \in \Gamma$,
the inequality (2) holds.
\item[(iii)] there exists a set $E_3 \subset [0,\infty)$ that has finite linear measure, such that for all z satisfying $|z|$ does not belongs to $E_3$ and for all $(k, j) \in \Gamma$, we have \begin{equation}\label{f"byf-2}
\left|\frac{f^{(k)}(z)}{f^{(j)}(z)}\right| \leq |z|^{(k-j)(\rho+\epsilon)}.
\end{equation}
\end{enumerate}
\end{lemma}
\begin{lemma}\cite{fuchs}\label{fuchs}
Let $f(z)$ be a meromorphic function of finite order $\rho$. Given $\xi> 0$ and $\delta$, $0 <\delta<1/2$, there is a constant $K(\rho,\xi)$ such that for all $r$ in a set $F$ of lower logarithmic density greater than $1-\xi$ and for every interval $J$ of length $\delta$ $$r\int\limits_J\left|\frac{f'(re^{\iota\theta})}{f(re^{\iota\theta})}\right|d\theta < K(\rho,\xi)(\delta\log\frac{1}{\delta})T(r,f).$$
\end{lemma}
\begin{lemma}\cite{myself}\label{myself}
Let $f(z)$ be an entire function satisfying $T(r,f)\sim\log M(r,f)$ in a set $E$ of positive upper logarithmic density. For given $0<c<\frac{1}{4}$ and $r\in E$, the set $$I_r=\{\theta\in [0,2\pi): \log |f(re^{\iota\theta})|\leq (1-c)\log M(r,f)\}$$ has linear measure zero. 
\end{lemma}
Next lemma is extracted from proof of the theorem in \cite{kwonkim}
\begin{lemma}\label{kwonkim}Let $f(z)$ be an entire function of finite order satisfying $T(r,f)\sim \log M(r,f)$ in a set $\overline{\log dens}E >0$ and satisfies condition of Lemma 4 in a set $F$ such that $\underline{\log dens}(F)>1-\xi$, then $$|f(z)|>(1-2c)\log M(r,f),$$ where $0<c<1/4$, $r\in E\cap F$ and $\theta\in [0,2\pi]\setminus I_r$ where $$I_r=\{\theta\in [0,2\pi): \log |f(re^{\iota\theta})|\leq (1-c)\log M(r,f)\}.$$ 
\end{lemma}
Let $$M(r,f) = max.\{|f(z)|:|z|=r\}$$ and $L(r,f) = min.\{|f(z)|:|z|=r\}$, then following lemma gives relation between $M(r,f)$ and $L(r,f)$ when $f$ has at most finite number of poles.
\begin{lemma}\cite{zheng}\label{zheng}
Let f be a transcendental meromorphic function with at most finitely many poles. If $\mathcal{J}(f)$ has only bounded components, then for any complex number $a$,  there exist a constant $0<d<1$ and two sequences $\{r_n\}$ and $\{R_n\}$ of positive numbers with $r_n\to\infty$ and $R_n/r_n\to\infty(n\to\infty)$ such that
$$M(r,f)^d\leq L(r,f), r\in G$$ where $G=\cup_{n=1}^\infty\{r:r_n<r<R_n\}.$  
\end{lemma}

\section{Proof of the theorems}
Before the proof of each of our result, we give some lemmas which will be useful in proving our results.\\

\textbf{Proof of Theorem 3.}

Let $\alpha <\beta$ be such that $\beta-\alpha<2\pi$, $r > 0$ and $\overline{S}$ denote the closure of $S$. Denote
$$S(\alpha , \beta )= \{z: \alpha < argz <\beta \}$$, 
$$S(\alpha,\beta, r) = \{z : \alpha < arg z < \beta\} \cap \{z : |z| < r\}.$$
 Let $f$ be an entire function of order $\rho(f) \in (0,\infty)$. For simplicity, set $\rho = \rho(f)$ and $S = S(\alpha ,\beta)$. Then $f$ is said to blow up exponentially in $\overline{S}$ if for any $\theta\in (\alpha , \beta)$
$$\lim\limits_{r\to\infty}\frac{\log\log |f(re^{\iota\theta})|}{\log r}=\rho$$ and decays to zero exponentially in $\overline{S}$ if for any  $\theta\in (\alpha , \beta )$
$$\lim\limits_{r\to\infty}\frac{\log\log |f(re^{\iota\theta})|^{-1}}{\log r}=\rho$$
\begin{lemma}\cite{hille}\label{hille}
Let $f$ be a non-trivial solution of $f'' + P(z)f = 0$, where $P(z) =
a_nz^n + \cdots + a_0$, $a_n \neq 0$. Set $\theta_j = \frac{2j\pi-arg(a_n)}{n+2}$ and $S_j = (\theta_j, \theta_{j+1})$, where $j = 0, 1, 2, \cdots , n + 1$ and $\theta_{n+2} = \theta_0 + 2\pi$. Then $f$ has the following properties:
\begin{enumerate}[(i)]
\item
In each sector $S_j$ , $f$ either blows up or decays to zero exponentially.
\item 
If, for some $j$, $f$ decays to zero in $S_j$ , then it must blow up in $S_{j-1}$ and $S_{j+1}$. However, it is possible for $f$ to blow up in many adjacent sectors.
\item
If f decays to zero in $S_j$ , then $f$ has at most finitely many zeros in any
closed sub-sector within $S_{j-1}\cup  S_j \cup S_{j+1}$.
\item
If $f$ blows up in $S_{j-1}$ and $S_j$, then for each $\epsilon > 0$, $f$ has finitely many zeros in each sector $\overline{S}(\theta_j - \epsilon, \theta_j + \epsilon)$, and furthermore, as $r \to \infty$,
$$n(\overline{S}(\theta_j - \epsilon, \theta_j + \epsilon, r), 0, f) = (1 + o(1))\frac{2\sqrt{|a_n|}}{\pi (n+2)}r^{\frac{n+2}{2}},$$
where $n(\overline{S}(\theta_j - \epsilon, \theta_j + \epsilon, r), 0, f)$ is the number of zeros of $f$ in the region $\overline{S}(\theta_j - \epsilon, \theta_j + \epsilon,r)$.

\end{enumerate}
\end{lemma}
\begin{lemma}\cite{gundersen2}\label{gundersen2}
Let $A(z)$ and $B(z)\not\equiv 0$  be two entire functions such that for real
constants $\alpha >0$, $\beta >0$, $\theta_1$, $\theta_2$, where $\alpha > 0$, $\beta > 0$ and $\theta_1 < \theta_2$, we have
\begin{equation}
|A(z)|\geq \exp \{(1+o(1)\alpha |z|^{\beta})\},
\end{equation}
\begin{equation}
|B(z)|\leq \exp \{(1+o(1)|z|^{\beta})\}
\end{equation}
as $z\to\infty$ in $\overline{S}(\theta_1, \theta_2) = \{z: \theta_1 \leq arg z \leq \theta_2\}$. Let $\epsilon > 0$ be a given small constant, and let $\overline{S}(\theta_1 + \epsilon, \theta_2 - \epsilon) = \{z: \theta_1+\epsilon \leq arg z \leq \theta_2\ -\epsilon\}$.
If $f$ is a non-trivial solution of \eqref{cde1} with $\rho(f) < \infty$, then the following conclusions hold:
\begin{enumerate}[(i)]
\item
There exists a constant $b(\neq 0)$ such that $f(z)\to b$ as $z\to \infty$ in $\overline{S}(\theta_1 + \epsilon, \theta_2 - \epsilon)$. Furthermore,
\begin{equation}
|f(z) - b| \leq exp\{-(1 + o(1))\alpha |z|^\beta \}
\end{equation}
as $z \to\infty$ in $\overline{S}(\theta_1 + \epsilon, \theta_2 - \epsilon)$.
\item
For each integer $k > 1$,
\begin{equation}
|f^{(k)}(z)| \leq exp\{-(1 + o(1))\alpha |z|^\beta \}
\end{equation}
as $z \to\infty$ in $\overline{S}(\theta_1 + \epsilon, \theta_2 - \epsilon)$.
\end{enumerate}
\end{lemma}

\begin{proof}
If $\rho(A) <\rho(B)$, then it is already proved in \cite{gundersen2} that all non-trivial solutions of \eqref{cde1} are of infinite order. Now, suppose that $\rho(A) >\rho(B)$.\\
Assume $\rho(f)<\infty$.
Set $\theta_j = \frac{2j\pi-arg(a_n)}{n+2}$ and $S_j = (\theta_j, \theta_{j+1})$, where $j = 0, 1, 2, \cdots , n + 1$ and $\theta_{n+2} = \theta_0 + 2\pi$.
Let $\theta\in S_j$. Since $A(re^{\iota\theta})$ is a solution of equation \eqref{cde2}, thus $A(z)$ either blows up or decays to zero exponentially in each sector $S_j$.\\
\textbf{Case 1:}
Suppose $A(z)$ blows up exponentially in each sector $S_j$. Then, we have $$\lim\limits_{r\to\infty}\frac{\log\log |A(re^{\iota\theta})|}{\log r}=\frac{n+2}{2}.$$
Then for any given constant $\epsilon \in (0, \frac{\pi}{\rho(A)})$ and $\beta \in (0,\frac{\rho(A) - \rho(B)}{3})$, we have
\begin{align*}
 |A(z)| & \geq \exp\{|z|^{\frac{n+2}{2} -\beta }\}\\
& \geq \exp\{\frac{1}{2}|z|^{\frac{n+2}{2} -\beta}+\frac{1}{2}|z|^{\frac{n+2}{2} -\beta}\}\\
& \geq \exp\{\frac{1}{2}|z|^{\frac{n+2}{2} -\beta}+\frac{1}{2}|z|^{\frac{n+2}{2} -2\beta}\}\\
& \geq \exp\{(1+|z|^{-\beta})\frac{1}{2}|z|^{\frac{n+2}{2} -\beta}\}\\
& \geq \exp\{(1+o(1))\alpha |z|^{\frac{n+2}{2} -\beta}\}
 \end{align*}

where $\alpha = \frac{1}{2}.$
$$|B(z)|\leq \exp\{|z|^{\rho(B) + \beta}\} \leq \exp\{|z|^{\rho(A) -2\beta}\} \leq \exp\{o(1)|z|^{\frac{n+2}{2} - \beta}\}$$
as $z\to\infty$ in $S_j(\epsilon) = \{z: \theta _j + \epsilon < arg.z < \theta_j - \epsilon\}$, $j = 0, 1,\cdots ,  n + 1$. 
 Combining above inequalities for $A(z)$, $B(z)$ and Lemma \ref{gundersen2}, exist corresponding constants $b_j \neq 0$ such that
$$|f(z) - b_j| \leq exp\{-(1 + o(1))\alpha|z|^{\frac{n+2}{2}-\beta}\}$$
as $z\to\infty$ in $S_j(2\epsilon)$, $j = 0, 1,\cdots, n + 1$. Therefore, $f$ is bounded in the whole complex plane by the Phragmen-Lindelof principle. So $f$ is a nonzero constant in the whole complex plane by Liouville's theorem. But $f$ cannot be non-constant, which gives rise to a contradiction.  
\\
\textbf{Case 2:}
Suppose $A(z)$ decays to zero in at least one sector $S_j$. Then, we have $$\lim\limits_{r\to\infty}\frac{\log\log |A(re^{\iota\theta})|^{-1}}{\log r}=\frac{n+2}{2}.$$ 
Then we get 
$$|A(re^{\iota\theta})| < exp (-r^{\frac{n+2}{2}-\xi})$$ where $r\to\infty$ and $\xi$ is a positive constant. Since $\rho(f)<\infty$
 by Lemma \ref{gundersen}, we have 
$$\left|\frac{f^{(k)}(z)}{f(z)}\right|\leq |z|^{k\rho(f)}$$
for $|z|\geq R=R(\theta)>0$ and $\theta\in[0,2\pi)/G$, where $G$ is a set with linear measure $0$.\\
Since we have $$T(r,B)\sim \log M(r,B)$$ in a set $E$ having positive upper logarithmic density, by using Lemma \ref{kwonkim}, for $0<c<1/4$ we have 
\begin{equation}\label{eqself}
M(r,B)^{1-2c}<|B(z)|
\end{equation}
where $r\in E\cap F$ and $\theta\in[0,2\pi]\setminus I_r$ where $I_r$ and $F$ as defined in Lemma \ref{kwonkim}.\\
From equation 1 we get
$$|B(z)|\leq\left|\frac{f''(re^{\iota\theta})}{f(re^{\iota\theta})}\right|+|A(z)|\left|\frac{f'(re^{\iota\theta})}{f(re^{\iota\theta})}\right|$$
$$M(r,B)^{1-2c}<|B(z)|\leq (1+o(1))r^{2\rho(f)}$$
$$M(r,B)<(1+o(1))r^{4\rho(f)}$$
for $r>R(\theta)$; $r\in  E\cap F \cap G $ and $\theta\in S_j\setminus I_r$.\\
But $M(r,B)<(1+o(1))r^{4\rho(f)}$ is not possible for transcendental entire function $B(z)$.\\
Hence $\rho(f)$ is infinite.

\end{proof}
{\bf{Proof of Theorem 5.}}\\
Following lemma provides the information of the minimum modulus of entire function of non-integral order having zeros in definite sectors.
\begin{lemma}\cite{kwon}\label{kwon}
Let $f(z)$ be an entire function of finite non-integral order $\rho$ and of genus $p>1$. Suppose that for any given $\epsilon>0$, all the zeros of $f(z)$ have their arguments in the following subset of real numbers:
$$S(p,\epsilon)=\displaystyle{\{\theta:|\theta|\leq\frac{\pi}{2(p+1)}-\epsilon\}}$$ if $p$ is odd, and
$$S(p,\epsilon)=\displaystyle{\{\theta:\frac{\pi}{2p}+\epsilon\leq|\theta|\leq\frac{3\pi}{2(p+1)}-\epsilon\}}$$ if $p$ is even. Then for any $c>1$, there exists a real number $R>0$ such that $$|f(-r)|\leq\exp(-cr^p)$$
for all $r\geq R$.
\end{lemma}
\begin{proof}
Let us suppose that solution $f$ of \eqref{cde1} is of finite order.
Using Lemma \ref{gundersen}(b), we have 
\begin{equation}\label{eqgundersen}
\left|\frac{f^{(k)}(re^{\iota \theta})}{f(re^{\iota \theta})}\right| \leq r^{k\rho(f)} 
\end{equation}
for $r\not\in E\cup [0,1]$, where $E$ is a set with finite logarithmic measure  and $r>R_1(\theta)$. Let us rotate the axes of the complex plane, assume that all the zeros of $A(z)$ have their arguments in the set $S(p,\epsilon)$ defined in Lemma \ref{kwon} for some $\epsilon>0$. Hence by Lemma \ref{kwon}, there exists a positive real number $R_2$ such that for all $r>R_2$, we have
\begin{equation}\label{eqkwon}
\min_{|z|=r}|A(z)|\leq|A(-r)|\leq\exp(-cr^p)< 1.
\end{equation}
 By Lemma \ref{zheng}, we have
\begin{equation}\label{eqzheng}
M(r,B)^d\leq |B(re^{\iota\theta})|
\end{equation}
for $0<d<1$ and $r\in G=\cup_{n=1}^\infty\{r:r_n<r<R_n\}$.
From equation \eqref{cde1} we get,
$$|B(z)|\leq\left|\frac{f''(re^{\iota\theta})}{f(re^{\iota\theta})}\right|+|A(z)|\left|\frac{f'(re^{\iota\theta})}{f(re^{\iota\theta})}\right|$$
Using \eqref{eqgundersen}, \eqref{eqkwon} and \eqref{eqzheng} for $r>max\{R_1(\theta),R_2\}$ such that $r\in G\setminus E_1\cup[0,1]$ and $\theta\in\{\theta:\min_{|z|=r}|A(z)|=|A(z)|\}$, we have
$$M(r,B)^{d}<|B(z)|\leq (1+o(1))r^{2\rho(f)}$$
$$M(r,B)<(1+o(1))r^{2\rho(f)}$$
which is a contradiction for a transcendental entire function.

\end{proof}   
\textbf{Proof of Theorem 6.}

Combining Theorem 2 from \cite{fuchs} and Lemma 2.2 from \cite{ishizaki}, we get
\begin{lemma}\label{fi}
Let $g(z) =\sum\limits_{n=0}^\infty a_{\lambda_n}z^{\lambda_n}$ be an entire function of finite order with Fabry gap, and let $u(z)$ be an entire function with $\rho(u)  \in (0,\infty)$. Then for any given $\epsilon \in (0, \varsigma)$, where $\varsigma=min(1,\rho(u))$, there exists a set $K \subset (1,\infty)$ satisfying $\overline{\log dense}K \geq \xi$, where $\xi \in (0, 1)$ is a constant such that for all $|z| = r \in K$, $$\log M(r, u) > r^{\rho(u) -\epsilon},\   \log m(r, g) > (1-\xi) \log M(r, g),$$ where $M(r, u) = max\lbrace|u(z)| : |z| = r\rbrace ,\ m(r, g) = min\lbrace|g(z)| : |z| = r\rbrace$ and $M(r, g) = max\{ |g(z)| : |z| = r\}$.
\end{lemma}

\begin{lemma}\cite{kwon2}\label{kwon2}
Let $f(z)$ be a non-constant entire function. Then there exists a real number $R$ such that for all $r\geq R$ there exists $z_r$ with $|z_r| = r$ satisfying
\begin{equation}
\left|\frac{f(z_r)}{f'(z_r)}\right|\leq r.
\end{equation}
\end{lemma}
\begin{lemma}\cite{bank}\label{bank}
 Let $A(z) = h(z)e^{P(z)}$ be an entire function with $\lambda(A) < \rho(A) = n$, where $P(z)$ is a polynomial of degree $n$. Then for every $\epsilon > 0$ there exists $E\subset[0, 2\pi)$ of linear measure zero such that
\begin{enumerate}
\item[(i)] for $\theta \subset E^+\setminus E$ there exists $R > 1$ such that
\begin{equation}
|A(re^{\iota\theta})| \geq \exp((1 - \epsilon)\delta(P, \theta)r^n)
\end{equation}
for $r > R$.
\item[(ii)] for $\theta \subset E^- \setminus E$ there exists $R > 1$ such that \begin{equation} |A(re^{\iota\theta})| \leq \exp((1 - \epsilon)\delta(P, \theta)r^n)
\end{equation}for $r > R$.
\end{enumerate}
\end{lemma}

\begin{proof}
Suppose $f$ is a solution of finite order of equation (\ref{cde1}). Then by Lemma \ref{gundersen}
 we get 
 \begin{equation}\label{eqgund}
\left|\frac{f^{''}(re^{\iota\theta})}{f'(re^{\iota\theta})}\right| \leq r^{\rho(f)}
\end{equation}
where $\theta\in [0,2\pi)\setminus E_1$, where $E_1$ is a set of linear measure $0$ for $r\geq r_1(\theta)>0$.\\
Using Lemma \ref{kwon2}, we get
\begin{equation}\label{eqkwon2}
\left|\frac{f(z_r)}{f'(z_r)}\right|\leq r
\end{equation}
for $r\geq r_2$ such that $|z_r|=r$.\\
Since $\lambda(B)<\rho(B)$, choosing $\theta\in [0,2\pi)\setminus E_2$, where $E_2$ is a set of linear measure $0$ such that $\delta(P,\theta)<0$ using Lemma \ref{bank}, we have
\begin{equation}\label{eqbank}
|B(re^{\iota\theta})| \leq \exp((1 - \epsilon)\delta(P, \theta)r^n)
\end{equation}
for $r>r_3$. From equation \eqref{cde1}, we get 
\begin{equation}\label{bycde}
|A(re^{\iota\theta})|\leq \left|\frac{f''(re^{\iota\theta})}{f'(re^{\iota\theta})}\right|+|B(re^{\iota\theta})|\left|\frac{f(re^{\iota\theta})}{f'(re^{\iota\theta})}\right|
\end{equation}
\begin{enumerate}[(a)]
\item Since $A(z)$ is a transcendental entire function with Fabry gap, using Lemma \ref{fi} there exist a set $H\subset (1,\infty)$ satisfying $\overline{\log dens}H \geq \xi$, where $\xi \in (0,1)$ for $r\in H$ such that
\begin{equation}\label{eqfi}
M(r,A)^{(1-\xi)}<|A(z)|.
\end{equation}
Using (\ref{eqgund}), (\ref{eqkwon2}), (\ref{eqbank}), (\ref{bycde}) and (\ref{eqfi}) we will get
$$M(r,A)^{(1-\xi)}\leq r^{\rho (f)}+r\exp((1 - \epsilon)\delta(P, \theta)r^n)$$ 
 $r\in H$, $r > max.\{r_1, r_2, r_3\}$ and $\theta\in [0,2\pi)\setminus (E_1 \cup E_2)$.\\
$$M(r,A)^{(1-\xi)}<r^{2\rho (f)}(1+o(1)),$$
which is a contradiction for very large $r$.
\item Using (\ref{eqself}), (\ref{eqgund}), (\ref{eqkwon2}), (\ref{eqbank}) and (\ref{bycde}) we will get
$$M(r,A)^{(1-2c)}<r^{2\rho (f)}(1+o(1)),$$
for $r\in E\cap F$, $r > max.\{r_1, r_2, r_3\}$ and $\theta\in [0,2\pi)\setminus (E_1 \cup E_2\cup I_r)$.\\
which is a contradiction for very large $r$.
\item Using (\ref{eqzheng}), (\ref{eqgund}), (\ref{eqkwon2}), (\ref{eqbank}) and (\ref{bycde}) for , we get for $0<d<1$, $r\in G=\cup_{n=1}^\infty\{r:r_n<r<R_n\}$, $r > max.\{r_1, r_2, r_3\}$ and $\theta\in [0,2\pi)\setminus E_1\cup E_2$.
$$M(r,A)^d<r^{2\rho (f)}(1+o(1)),$$
which is a contradiction for very large $r$.

\end{enumerate}
\end{proof}

\end{document}